\documentclass[a4paper, 11pt,reqno]{amsart}
\usepackage[T1]{fontenc}
\usepackage{lmodern}
\usepackage[utf8]{inputenc}
\usepackage{amsthm,amssymb,amsmath,graphicx,psfrag,enumitem,mathrsfs}
\usepackage{mathtools}
\usepackage[foot]{amsaddr}
\usepackage{amsfonts}
\usepackage{dsfont}
\usepackage[british]{babel}
\usepackage[babel]{microtype}

\usepackage[margin=1in]{geometry}

\usepackage[textsize=footnotesize]{todonotes}

\usepackage{algorithm}
\usepackage{tikz}
\usepackage{tikz-cd}

\usetikzlibrary{backgrounds}

\usepackage[%
  backend=bibtex,%
  bibstyle=numeric,%
  citestyle=numeric-comp,%
  sorting=nyvt,%
  natbib=true,%
  sortcites=true,%
  maxbibnames=99,%
  maxcitenames=5,%
  minnames=1,%
  autocite=plain,%
  uniquelist=false,%
  firstinits=true,%
  date=comp]%
    {biblatex}

\usepackage{hyperref}
\usepackage{cleveref}
\hypersetup{
    colorlinks=true,
    linkcolor=blue,
    filecolor=magenta,      
    urlcolor=cyan,
}

\allowdisplaybreaks
\newcommand{\COMMENT}[1]{}

\newtheorem{definition}{Definition}

\newtheorem{theorem}[definition]{Theorem}

\newtheorem{lemma}[definition]{Lemma}

 \newcommand{\R}{\mathbb R}

 \newcommand{\X}{\mathcal X}  
 \newcommand{\Y}{\mathcal Y}    
 \newcommand{\x}{\mathbf x}  
 \newcommand{\y}{\mathbf y} 
  
 \newcommand{\vv}{\mathbf v}                    

\title{Entropy versions of additive inequalities}

\author{Alberto Espuny D\'iaz}
\address{School of Mathematics, University of Birmingham}
\author{Oriol Serra}
\address{Department of Mathematics, Universitat Polit\`ecnica de Catalunya}
\email{AXE673@bham.ac.uk, oriol.serra@upc.edu}
\thanks{The research leading to these results was partially supported by the Spanish MECD through grant 2015/COLAB/00069 (A.~Espuny D\'iaz) and projects MTM2014-54745-P and MDM-2014-0445 (O.~Serra).}
\date{\today}

\begin{document}

\begin{abstract}
The connection between inequalities in additive combinatorics and analogous versions in terms of the entropy of random variables has been extensively explored  over the past few years.
This paper extends a device introduced by Ruzsa in his seminal work introducing this correspondence.
This extension provides a toolbox for establishing the equivalence between sumset inequalities and their entropic versions.
It supplies simpler proofs of known results and opens a path for obtaining new ones.
\end{abstract}
\maketitle
\thispagestyle{empty}

\section{Introduction}\label{section1}

In recent years, several authors realized that there exist certain analogies between many of the cardinality inequalities in additive combinatorics that have been developed over the years and some entropy inequalities.
These analogies appear, for instance, with many important sumset inequalities such as the Pl\"unnecke-Ruzsa inequalities, or with traditional entropy results such as Shearer's inequality.
In the past decade, several papers exploring these analogies have appeared and many insightful results have been produced. 
The seminal work of \citet{RuzsaEntropy} on this topic was extended by \citet{BalisterBollobas}, \citet{Madimanmutual}, \citet{MadimanGeneral}, \citet{Madimanjointentropies} or \citet{TaoEntropy}, among many others.
All these papers present different techniques with which the analogy between sumset inequalities and entropy inequalities can be studied.
These techniques are used to obtain many new results, especially in the form of entropy inequalities.

In this note we concern ourselves with entropies of discrete random variables.
Let $X$ be a discrete random variable taking values $x_1,x_2,\ldots,x_n$ with probabilities $p_1,p_2,\ldots,p_n$, respectively.
The Shannon entropy of $X$ is defined as 
\[
\mathbf{H}(X)\coloneqq\sum_{i=1}^np_i\log\frac{1}{p_i}.
\] 
The definition is analogous if $X$ takes countably many values.
This is a concave function, and Jensen's inequality gives
\begin{equation}\label{equa:jensen}
\mathbf{H}(X)\leq\log n,
\end{equation}
where $n$ is the cardinality of the range of $X$.
Moreover, equality holds if and only if $X$ is uniformly distributed over its range.
This is the key property which allows one to translate entropy inequalities to combinatorial ones.
From this perspective, entropy inequalities can be seen as generalizations of their combinatorial counterparts.
One of the first examples in the literature is the translation of the classical inequality of Han,
\[
(n-1)\mathbf{H}(X_1,\ldots,X_n)\leq\sum_{i=1}^n\mathbf{H}(X_1,\ldots,X_{i-1},X_{i+1},\ldots,X_n),
\]
which provides a simple direct proof of the inequality of Loomis and Whitney
\[
|A|^{n-1}\leq\prod_{i=1}^n|A_i|,
\]
where $A\subset E_1\times \cdots \times E_n$ and  $A_i$ denotes the projection of $A$ to the $i$-th coordinate hyperplane.
This example opened the path to obtaining combinatorial inequalities from entropy ones. 

\citet{RuzsaEntropy} introduced a device to walk the path backwards and obtain entropy inequalities from combinatorial ones, by establishing in fact the equivalence between the two versions.
In his paper, he restricted the device to linear functions of two variables in abelian groups. Ruzsa used this technique to prove the equivalence between Han's inequality and the Loomis and Whitney theorem mentioned above.  
This same technique was later used by \citet{BalisterBollobas} to prove the equivalence between Shearer's inequality and the Uniform Covering inequality.

The main goal of this note is to extend the device of Ruzsa to arbitrary functions.
By doing so we obtain a more flexible tool.
Our hope is that this will allow to give new combinatorial proofs of entropy inequalities, and also to obtain new ones.

Given a function $f\colon\mathcal{X}\to\mathcal{Y}$, we denote by $f^k$ the function $f^k\colon\mathcal{X}^k\to \mathcal{Y}^k$ induced on the $k$-fold cartesian power $\mathcal{X}^k$, namely, $f^k(x_1,\ldots,x_k)=(f(x_1),\ldots,f(x_k))$ for $x_1,\ldots,x_k\in\mathcal{X}$.
The main result of this note is the following one.

\begin{lemma}\label{teor:cardtoentro}
	Let\/ $f,f_1,\ldots,f_n$ be  functions defined over a set\/ $\mathcal{X}$.
	Let\/ $\alpha_1,\ldots,\alpha_n$ be real numbers.
	If for all positive\/ $k$ and every finite set\/ $A\subseteq\mathcal{X}^k$ we have that
	\[|f^k(A)|\leq\prod_{i=1}^n\left|f_i^k(A)\right|^{\alpha_i},\]
	then, for every discrete random variable\/ $X$ taking values in\/ $\mathcal{X}$, the entropy of\/ $f(X)$ satisfies
	\[\mathbf{H}(f(X))\leq\sum_{i=1}^n\alpha_i\mathbf{H}(f_i(X))\]
	whenever\/ $\mathbf{H}(f_i(X))$ is finite for every\/ $i\in[n]$.
\end{lemma}

\Cref{teor:cardtoentro} is complemented by the following partial converse, which can be obtained from the concavity of the entropy function.

\begin{lemma}\label{teor:entrotocard}
	Let\/ $f,f_1,\ldots,f_n$ be any functions defined over a set\/ $\mathcal{X}$.
	Let\/ $\alpha_1,\ldots,\alpha_n$ be positive real numbers.
	If the inequality
	\[\mathbf{H}(f(X))\leq\sum_{i=1}^n\alpha_i\mathbf{H}(f_i(X))\]
	holds for every random variable $X$ with suport in a finite set $A\subseteq{\mathcal X}$, then we have that
	\[|f(A)|\leq\prod_{i=1}^n\left|f_i(A)\right|^{\alpha_i}.\]
\end{lemma}

We will first present the technique developed by Ruzsa in \cref{section2}, as well as its generalization.
We also provide proofs for \cref{teor:cardtoentro} and \cref{teor:entrotocard}.
\Cref{teor:cardtoentro} is actually shown through two technical lemmas, \cref{lema:ruzdevice} for random variables taking a finite number of values with rational probabilities, and an extension to discrete random variables in \cref{lema:ruzdevice2}.
%From \cref{section3} onwards we present several applications of our technique. 
%All of them result in equivalence theorems, new entropic inequalities or examples of the unified approach we present in the form of short proofs using the Ruzsa device as a black box.
%In \cref{section3} we present some of the more direct examples, considering sums and differences of a few sets or random variables.
In \cref{section35} we present an application of our technique to prove a new result which generalises the example of the equivalence between Han's inequality and the Loomis and Whitney inequality to fractional coverings. 
%In \cref{section4} we present applications to obtain Pl\"unnecke-Ruzsa type estimates on the growth of entropies of iterated operations of random variables.
%\Cref{section5} is devoted to obtaining analogous results in the non-commutative setting, including the Ruzsa triangle inequality.
%Entropic inequalities in this setting are very scarce in the literature. 

% Some of the results we present are new and many of them are known. 
% We emphasize that our contribution aims at giving simple unified proofs of all the results by stressing the connection between combinatorial and entropy inequalities, thus deepening the analogy introduced by \citet{RuzsaEntropy} and pursued by other authors. 
% This is particularly relevant in the obtention of entropy inequalities, which in most cases were derived from purely entropic arguments.

%%%%%%%%%%%%%%%%%%%%%%%%%%%%%%%%%%%%%%%%%%%%%%%%%%%%%%%%%%%%%%%%%%%%%%%%
%%%%%%%%%%%%%%%%%%%%%%%%%%%%%%%%%%%%%%%%%%%%%%%%%%%%%%%%%%%%%%%%%%%%%%%%

\section{The Ruzsa device}\label{section2}

Obtaining cardinality inequalities analogous to entropy inequalities, as in \cref{teor:entrotocard}, is based on the fact that uniform random variables capture the information of  their range sets. In order to invert the analogy and obtain entropy inequalities from cardinality ones, \citet{RuzsaEntropy} proposed a construction of sets which captures the probability distribution of a given random variable. 

Assume we are given a random variable $X$ defined over a set $\mathcal{X}$ that takes a finite number of values, each of them with a rational probability.
We can then construct a set $R_k(X)\subseteq\mathcal{X}^k$, to which we will refer as the $k$-\textit{Ruzsa set} of $X$, for infinitely many values of $k$.
The vectors in $R_k(X)$ have the property that, if a coordinate in one of them  is chosen uniformly at random, then we are choosing an element $x\in\mathcal{X}$ with the same probability as the random variable $X$ does (we may say that the ``density'' of $x$ in the vector equals the probability that it is the outcome of $X$).

\begin{definition}
Let\/ $X$ be a random variable taking values\/ $\{x_1,\ldots,x_n\}\subseteq \X$, each with probability\/ $\displaystyle p_i=\frac{q_i}{r_i}$ for some\/ $q_i,r_i\in\mathbb{N}$, and let\/ $k$ be a common multiple of\/ $r_1,\ldots,r_k$.
For any vector\/ $\vv=(v_1,\ldots,v_k)\in \X^k$ and each $i\in [n]$, let\/ $J_i(\vv)\coloneqq\{j\in [k] :v_j=x_i\}$.
The\/ $k$-\emph{Ruzsa set} of\/ $X$ is the set of vectors
\[
R_k(X)\coloneqq\{\vv\in  \X^k:  |J_i(\vv)|=p_ik\ \ \forall\ i\in[n]\}.
\]
We call an integer\/ $k$ \emph{suitable} for the random variable\/ $X$, or $X$\emph{-suitable}, if it is a common multiple of\/ $r_1,\ldots,r_k$.
\end{definition}

With this definition we have that
\[
|R_k(X)|=\binom{k}{p_1k,\ldots ,p_nk},
\]
and, by using  Stirling's formula, one can readily check that
\begin{equation}\label{equa:Acard}
\log|R_k(X)|=k\,\mathbf{H}(X)+O(\log k)\text{.}
\end{equation}

This is the construction Ruzsa used to prove the equivalence between Han's inequality and the Loomis and Whitney theorem.
The main idea behind the proof came from observing that one can build a set from a random variable, a different set from its projection onto a certain subspace, and that the resulting set in the latter is precisely the projection of the first one.
In other words, the following diagram is commutative (here, $\pi_i$ stands for the projection onto the $i$-th coordinate hyperplane):
\[
    \begin{tikzcd}[row sep=large]
        X \arrow[r,"R_k"] \arrow[d,"\pi_i"] 
        & R_k(X) \arrow[d,"\pi_i^k"]
        \\
        \pi_i(X) \arrow[r,"R_k"]
        & R_k(\pi_i(X)) 
    \end{tikzcd}
\]
Using this fact, one can separately compute the sizes of $R_k(X)$ and its projections in terms of the entropy of the random variables through \eqref{equa:Acard}.
If a relationship between the sizes of the set and its projections is known, a relationship between the entropies of the variable and its projections follows (by letting $k$ tend to infinity).

Ruzsa took the idea behind these commutative diagrams a bit further.
Instead of considering simple projections, he took linear functions defined over two variables, and again proved that constructing the Ruzsa set and applying linear functions commute.
He used this fact to prove an equivalence theorem between inequalities of cardinalities of sumsets along graphs and entropy inequalities.
In this paper we go even further, and see that the above diagram is always   commutative, no matter which function $f$ is considered.

We say that a random variable taking values in a set $\mathcal{X}$ is an \textit{$\mathcal{X}$-random variable}.
Let $X$ be a discrete $\mathcal{X}$-random variable that takes finitely many values, each of them with rational probability.
Let $k\in\mathbb{N}$ be suitable for $X$.
Let $R_k(X)\subseteq\mathcal{X}^k$ be $X$'s $k$-Ruzsa set.
Let $f\colon\mathcal{X}\to\mathcal{Y}$ be any function.
Let us denote by $f^k\colon\mathcal{X}^k\to\mathcal{Y}^k$ the function induced by $f$ on the $k$-fold cartesian power $\X^k$, namely $f^k(x_1,\ldots,x_k)=(f(x_1),\ldots,f(x_k))$ for all $x_1,\ldots,x_k\in\mathcal{X}$.
The diagram now looks as follows:
\[
    \begin{tikzcd}[row sep=large]
        X \arrow[r,"R_k"] \arrow[d,"f"] 
        & R_k(X) \arrow[d,"f^k"]
        \\
        f(X) \arrow[r,"R_k"]
        & f^k(R_k(X)) 
    \end{tikzcd}
\]

\begin{lemma}\label{prop:Ruzsadevice}
Let\/ $X$ be an\/ $\mathcal{X}$-random variable taking finitely many values, each of them with a rational probability, and let\/ $f\colon\mathcal{X}\to\mathcal{Y}$ be a function defined over\/ $\mathcal{X}$. Then, for each\/ $X$-suitable\/ $k$,  
\begin{equation}\label{eq:Ruzsadevice}
R_k (f(X))=f^k(R_k(X)).
\end{equation}
\end{lemma}  

\begin{proof}
Assume that $X$ takes values $\{x_1,\ldots,x_n\}$, each with probability $\displaystyle p_i=\frac{q_i}{r_i}$ for some $q_i,r_i\in\mathbb{N}$, and construct the $k$-Ruzsa set $R_k(X)$ for some suitable $k$. 

Let $\{y_1,\ldots ,y_m\}$ be the range of $Y=f(X)$.
Every value $y_i$ is taken by $Y$ with a rational probability $p'_i=q'_i/r'_i\coloneqq\sum_{x\in f^{-1}(y_i)} \Pr (X=x)$, where $\mathrm{lcm}(r'_1,\ldots ,r'_m)$ divides $\operatorname{lcm}(r_1,\ldots ,r_n)$, so that $k$ is suitable for $Y$ and we can construct the $k$-Ruzsa set of $Y$. 

The image by $f^k$ of a vector $\mathbf{x}\in R_k(X)$  is a vector in $\Y^k$ in which every $y_i\in f(\mathcal{X})$ appears precisely $\displaystyle k\!\!\!\!\sum_{x\in f^{-1}(y_i)}\!\!\!\!\Pr(X=x)=kp'_i$ times.
Hence $f^k(\x)\in R_k(Y)$ and $f^k(R_k(X))\subseteq R_k(f(X))$.

Reciprocally, let $\mathbf{y}$ be a vector in $R_k(f(X))$.
We now construct a vector $\mathbf{x}\in\mathcal{X}^k$ such that $f^k(\mathbf{x})=\mathbf{y}$.
For each $i\in [m]$ let $J_i(\mathbf{y})\coloneqq\{\ell\in [k] :\mathbf{y}_\ell=y_i\}$ (note that $|J_i(\mathbf{y})|=kp'_i$).
For each $y_i$ let $f^{-1}(y_i)=\{x_1^i,\ldots,x_{n_i}^i\}$.
Partition $J_i(\mathbf{y})$ into $n_i$ sets $J_1^i,\ldots,J_{n_i}^i$ such that $|J_j^i|=k\Pr(X=x_j^i)$ for all $j\in[n_i]$ (note that this can be done as $p'_i=\sum_{x\in f^{-1}(y_i)} \Pr (X=x)$ and $kp_\ell\in\mathbb{N}$ for all $\ell\in[n]$).
Construct $\mathbf{x}$ by adding, for each $i\in[m]$ and $j\in[n_i]$, $x_j^i$ to the coordinates whose indices are in $J_j^i$.
For this vector we have $f^k({\mathbf x})=\y$.
This shows that $R_k(f(X))\subseteq f^k(R_k(X))$.

Hence, we have $f^k(R_k(X))=R_k(f(X))$.
\end{proof}

Once we have shown that the  diagram is commutative, we can provide a proof of \cref{teor:cardtoentro}.
We first consider random variables with finite support taking their values with rational probabilities.

\begin{lemma}\label{lema:ruzdevice}
Let\/ $f,f_1,\ldots,f_n$ be any functions defined over a set\/ $\mathcal{X}$.
Let\/ $\alpha_1,\ldots,\alpha_n$ be real numbers.
Let\/ $X$ be a random variable with support in\/ $\mathcal{X}$ that takes a finite number of values, each of them with a rational probability, and let\/ $\mathit{Suit}(X)\subseteq\mathbb{N}$ be the set of all\/ $X$-suitable integers.
If we have that
\begin{equation}\label{equa:lema4cond}
|f^{ k}(R_{k}(X))|\leq\prod_{i=1}^n\left|f_i^{ k}(R_{ k}(X))\right|^{\alpha_i}\ \ \ \forall\ k\in\mathit{Suit}(X)\text{,}
\end{equation}
then the entropy of\/ $f(X)$ satisfies
\[\mathbf{H}(f(X))\leq\sum_{i=1}^n\alpha_i\mathbf{H}(f_i(X))\text{.}\]
\end{lemma}

\begin{proof} 
For each $k\in\mathit{Suit}(X)$, build the set $R_{k}(X)\subseteq\mathcal{X}^{k}$, and consider $f^{k}(R_{k}(X))\subseteq f(\mathcal{X})^{k}$ and $f_i^{k}(R_{k}(X))\subseteq f_i(\mathcal{X})^{k}$ for each $i\in[n]$.
By \cref{prop:Ruzsadevice}, these sets are the Ruzsa sets of $f(X)$ and $f_i(X)$ for each $i\in[n]$, respectively.

By the hypothesis of the statement and \eqref{equa:Acard}, for every suitable $k$ we have that
\[
\mathbf{H}(f(X))+O\left(\frac{\log k}{k}\right)\leq\sum_{i=1}^n\alpha_i \mathbf{H}(f_i(X))+O\left(\frac{\log k}{k}\right).
\]
The conclusion follows by letting $k$ tend to infinity.
\end{proof}

A standard limit  procedure extends \cref{lema:ruzdevice} to general discrete random variables. 

\begin{lemma}\label{lema:ruzdevice2} 
If the hypothesis of \cref{lema:ruzdevice} hold for every random variable taking a finite number of values with rational probabilities, then its conclusion also holds for any discrete random variable\/ $X$ such that the entropies\/ ${\mathbf H}(f_i(X))$ are all finite.
\end{lemma}                                                                                                                                                      

\begin{proof} 
We can write $X$ as the limit of a sequence $X_i$ of random variables taking a finite number of values with rational probabilities. 
By \cref{lema:ruzdevice}, for each $i>0$ we have that $\mathbf{H}(f(X_i))\leq\sum_{j=1}^n\alpha_j\mathbf{H}(f_j(X_i))$.
As the discrete random variables  $X_i$ converge to $X$ in distribution, the corresponding entropies also converge and 
\[\mathbf{H}(f(X))=\lim_{i\to\infty}\mathbf{H}(f(X_i))\leq\lim_{i\to\infty}\sum_{j=1}^n\alpha_j\mathbf{H}(f_j(X_i))=\sum_{j=1}^n\alpha_j\mathbf{H}(f_j(X)).\qedhere\]
\end{proof}

Let $\mathfrak{X}$ be the set of all random variables taking finitely many values in $\mathcal{X}$, each of them with rational probability. 
By \cref{lema:ruzdevice2} we can restrict the proofs of our statements to random variables in $\mathfrak{X}$. 
We will use this fact in all the proofs.

Note that \cref{lema:ruzdevice2} is in fact stronger than \cref{teor:cardtoentro}, in the sense that the inequalities \eqref{equa:lema4cond} in the former are only required to hold for Ruzsa sets.  
However, in the applications we will usually have the more restrictive conditions.
The conditions in \cref{teor:cardtoentro} trivially imply those imposed in \cref{lema:ruzdevice2}.

Finally, for the sake of completeness, we provide a proof for \cref{teor:entrotocard}.

\begin{proof}[Proof of \cref{teor:entrotocard}]
For the proof one needs to define an appropriate random variable $X$.
Consider $f(A)$ and, for each $b\in f(A)$, choose a unique representative $a^*\in f^{-1}(b)$ of its preimage.
Let the set of these representatives be $A^*$, so that $f(A^*)=f(A)$.
Define a random variable $X$ having probability $\frac{1}{|f(A)|}$ of taking each value in $A^*$, and zero probability otherwise.
Thus $f(X)$ is uniformly distributed over $f(A)$. By \eqref{equa:jensen},
\begin{equation}\label{equa:entrotocard}
\log|f(A)|=\mathbf{H}(f(X))\leq\sum_{i=1}^n\alpha_i\mathbf{H}(f_i(X))\leq\sum_{i=1}^n\alpha_i\log|f_i(A)|\text{,}
\end{equation}
as it is clear that $f_i(A^*)\subseteq f_i(A)$ for all $i\in\{1,\ldots,n\}$.
\end{proof}

The reason that the numbers $\alpha_1,\ldots,\alpha_n$ have to be positive in \cref{teor:entrotocard} is that the inequality in \eqref{equa:entrotocard} is not guaranteed to hold otherwise.
However, the proofs of \cref{lema:ruzdevice,lema:ruzdevice2} also hold when these values are negative.
The fact that we have negative coefficients is what prevents us from directly proving many sumset versions of entropic results.
We observe that this same problem extends to the general use of Ruzsa's device.\COMMENT{Do we have any example where negative coefficients fail? }

There is an additional reason which prevents from a straight application of   \cref{teor:entrotocard}. Some entropy inequalities hold under independence constraints; this may  render \cref{teor:entrotocard} ineffective, as the proof relies strongly on some non independent random variables.
Likewise, when using \cref{lema:ruzdevice,lema:ruzdevice2}, there are   applications in which set cardinality inequalities for the Ruzsa sets of random variables can only hold (or are only known to hold) when the random variables involved are independent.
In these cases, an equivalence theorem using this method may not be possible.

\section{Projections and entropy}\label{section35}

As we already mentioned in the Introduction, Ruzsa proved that the Loomis and Whitney inequality and Han's inequality are, in fact, equivalent.
\citet{BalisterBollobas} generalised this result.
In order to state it, we must introduce some notation.
Let $B_1,\ldots,B_n$ be arbitrary sets, and let $\mathcal{X}=B_1\times\ldots\times B_n$.
Given a set $A\subseteq\mathcal{X}$, we denote its projection to the $i$-th coordinate by $A_i$ and, in general, its projection to coordinates indexed by $S\subseteq[n]$ as $A_S$.
In particular, $A=A_{[n]}$.
The same notation holds for random variables $X=(X_1,\ldots ,X_n)$ taking values in $\mathcal{X}$: for each subset $S\subset [n]$ we denote $X_S\coloneqq(X_i: i\in S)$.

A \textit{$k$-cover} of $[n]$ is a multiset $\mathcal{S}$ of subsets of $[n]$ such that each $i\in[n]$ appears in at least $k$ members of $\mathcal{S}$.
If each $i\in[n]$ appears in exactly $k$ members of $\mathcal{S}$, we say that the $k$-cover is \textit{uniform}.
Balister and Bollob\'as provided the following equivalence.

\begin{theorem}[\cite{BalisterBollobas}]
Let $n\geq2$, $B_1,\ldots,B_n$ be arbitrary finite sets, and let $\mathcal{X}\coloneqq B_1\times\ldots\times B_n$. For every uniform $k$-cover $\mathcal{S}$ of $[n]$,
the following two statements hold and are equivalent:
\begin{enumerate}[label=(\roman*)]
\item\label{item1BB} for any set $A\subseteq\mathcal{X}$,
 \[|A|^k\leq\prod_{S\in\mathcal{S}}|A_S|.\]
\item\label{item2BB}  for any random variable $X$ taking values in $\mathcal{X}$, 
 \[k\mathbf{H}(X)\leq\sum_{S\in\mathcal{S}}\mathbf{H}(X_S).\]
\end{enumerate}
\end{theorem}

We recall that  the first result is the well-known uniform cover inequality~\cite{unifcover}, whereas the second is the famous Shearer's inequality~\cite{Shearer}. What Balister and Bollob\'as proved is their equivalence.

In this same spirit, we show a further generalisation of this equivalence result using an entropic inequality of \citet{Madimanjointentropies}.
In order to state this result, we introduce some further notation.
Given a multiset $\mathcal{S}$ of subsets of $[n]$, a function $\alpha\colon\mathcal{S}\to\mathbb{R}_+$  is called a \textit{fractional cover} if for each $i\in[n]$ we have that
\[\sum_{S\in\mathcal{S}:i\in S}\alpha(S)\geq1\text{.}\]

As an analogy with random variables, given sets $A\subseteq\mathcal{X}$, $S\subseteq[n]$ and $y\in A_S$, we define the set $A\mid A_{S}=y$ (and read $A$ \textit{conditioned to} $A_S=y$) as the subset of $A$ such that the coordinates indexed by $S$ take the value $y$.
We then define an average size of the conditioned set by
\[\big|A\mid A_S\big|\coloneqq\prod_{y\in A_S}\big|A\mid A_S=y\big|^{p(y)},\]
where $p(y)$ is the probability that $\pi_S(x)=y$ when choosing a point $x\in A$ uniformly at random.

Given two sets $S,T\subseteq[n]$, we write $A_T\mid A_S$ to mean $A_T\mid A_{S\cap T}$, with a slight abuse of notation. The average size of $A_T\mid A_S$ is then
\[
\big| A_T\mid A_S\big|=\prod_{y\in A_S} \big| A_T\mid A_S=y\big|^{p(y)}.
\]
By convention, when $S=\varnothing$ we set $|A_T\mid A_{\varnothing}|=|A_T|$.

For a set $S\subseteq[n]$ with minimal element $a\geq1$,
we define $S_*\coloneqq[a-1]$. When $a=1$, we understand that $S_*$ is empty.

\begin{theorem}\label{thm:projection}
Let\/ $n\geq2$,\/ $B_1,\ldots,B_n$ be arbitrary finite sets, and let\/ $\mathcal{X}\coloneqq B_1\times\cdots\times B_n$. 
Let $\mathcal{S}$ be a multiset of\/ $[n]$. 
For any fractional covering $\alpha\colon{\mathcal{S}}\to \R_+$ the following statements hold and are equivalent:
\begin{enumerate}[label=(\roman*)]
\item\label{item1newBB} for any set\/ $A\subseteq\mathcal{X}$, 
 \[|A|\leq\prod_{S\in\mathcal{S}}\big|A_S\mid A_{S_*}\big|^{\alpha_S}.\]
\item\label{item2newBB}  for any random variable\/ $X=(X_1,\ldots,X_n)$ taking values in\/ $\mathcal{X}$ such that\/ $\mathbf{H}(X)$ is finite,
 \[\mathbf{H}(X)\leq\sum_{S\in\mathcal{S}}\alpha_S\mathbf{H}(X_S\mid X_{S_*}).\]
\end{enumerate}
\end{theorem}

\begin{proof}
Statement \ref{item2newBB} is a  result of \citet[Theorem I']{Madimanjointentropies}.
Statement \ref{item1newBB} follows from \ref{item2newBB} by taking $f$ to be the identity and $f_S$ to be projections onto the coordinates indexed by $S\in\mathcal{S}$ in \cref{teor:entrotocard}.

For the converse implication, choose $f$ to be the identity and let $f_S$ be the projections onto the coordinates in $S$.
Assume first that $X$ takes finitely many values, each with a rational probability. 
For each $X$-suitable $k$, $f^k_S(R_k(X))=R_k(X_S)$ by \cref{prop:Ruzsadevice}. 
By statement~\ref{item1newBB}, $|f^{k}(R_{k}(X))|\leq\prod_{S\in\mathcal{S}}|f_S^{k}(R_{k}(X))|^{\alpha_S}$ holds, so \cref{lema:ruzdevice} directly yields the result.
The case when $X$ is any discrete random variable follows by \cref{lema:ruzdevice2}.
\end{proof}

We note that statement \ref{item1newBB} in \cref{thm:projection} generalizes and improves previous bounds on the sizes of sets based on the sizes of their projections. The result can be derived from the more general results by \citet{Madimanjointentropies}.

\section*{Acknowledgements}

We would like to thank Mark Butler, Andrea Pachera and Jack Saunders for some very helpful discussions.
We are also indebted to Jiange Li for finding a mistake in a previous version of this note.

%%%%%%%%%%%%%%%%%%%%%%%%%%%%%%%%%%%%%%%%%%%%%%%%%%%%%%%%%%%%%%%%%%%%%%%%
%%%%%%%%%%%%%%%%%%%%%%%%%%%%%%%%%%%%%%%%%%%%%%%%%%%%%%%%%%%%%%%%%%%%%%%%

\printbibliography

\end{document}